\documentclass[12pt,reqno]{article}
\usepackage{enumerate}
\usepackage{fullpage}
\usepackage{amsmath, amssymb, amsthm, url, graphicx}

\newtheoremstyle{plainsl}%
        {\topsep}
        {\topsep}
        {\slshape} % only non-default setting
        {}
        {\normalfont\bfseries}
        {.}
        { }
        {}

\theoremstyle{plainsl}
\newtheorem{theorem}{Theorem}[section]
\newtheorem{proposition}[theorem]{Proposition}
\newtheorem{lemma}[theorem]{Lemma}
\newtheorem{corollary}[theorem]{Corollary}

\newtheorem{example}[theorem]{Example}

\newtheorem{remark}[theorem]{Remark}

\newcommand\tref[1]{Theorem~\ref{thm:#1}}
\newcommand\cref[1]{Corollary~\ref{cor:#1}}
\newcommand\sref[1]{Section~\ref{sec:#1}}
\newcommand\pref[1]{Proposition~\ref{pro:#1}}

\newcommand\sqr[2]{{\vbox{\hrule height.#2pt
    \hbox{\vrule width.#2pt height#1pt \kern#1pt
        \vrule width.#2pt}\hrule height.#2pt}}}
\renewcommand\qed{%
        \ifmmode\eqno\sqr53
        \else\nolinebreak\ \hfill\sqr53\medbreak\fi}

\numberwithin{equation}{section}
\newcommand{\m}{{{24}}}  % Special for this paper, subscripts and superscripts {24}
\newcommand{\ip}[1]{\langle {#1}\rangle}

\newcommand{\re}{{\mathbb R}}
\newcommand{\cx}{{\mathbb C}}
\newcommand{\ints}{{\mathbb Z}}
\renewcommand{\i}{{\rm i}}
\newcommand{\BMA}{{\mathbb A}}

\newcommand{\E}{{\mathsf E}}
\newcommand{\cE}{{\mathcal E}}  % Euclidean space
\newcommand{\F}{{\mathcal F}}
\newcommand{\G}{{\mathcal G}}
\newcommand{\I}{{\mathsf I}}
\newcommand{\J}{{\mathsf J}}
\renewcommand{\L}{\boldsymbol{\wedge}}  % Leech lattice in \sf font
\newcommand{\Z}{{\mathsf Z}}
\newcommand{\cI}{{\mathcal I}}

\newcommand{\R}{{\mathcal R}}
\renewcommand{\S}{{\mathsf S}}
\newcommand{\cS}{{\mathcal S}}

\newcommand{\cZ}{{\mathcal Z}}

\newcommand{\Nm}{{\sf Nm}}

\newcommand{\one}{{\mathbf 1}}

\DeclareMathOperator\Rad{Rad}
\DeclareMathOperator\rank{rank}
\DeclareMathOperator\spn{span}
\DeclareMathOperator\colsp{colsp}

\DeclareMathOperator\Jac{Jac}
\DeclareMathOperator\mult{mult}

\begin{document}
\thispagestyle{empty}
\setcounter{page}{1}
\title{On the ideal of the shortest vectors in the Leech lattice and other lattices}
\author{William J. Martin  \\
Corre L.~Steele \\
Department of Mathematical Sciences \\
Worcester Polytechnic Institute \\
Worcester, Massachusetts \\
{\tt \{martin,clsteele\}@wpi.edu}}

\date{August 2, 2014}
\maketitle

\medskip

\begin{abstract}
Let $X \subset \re^m$ be a spherical code (i.e., a finite subset of the unit sphere) and consider  
the ideal of all polynomials in $m$ variables which vanish on $X$. Motivated by a study of cometric ($Q$-polynomial)
association schemes and spherical designs, we wish to determine certain properties of this ideal. After presenting
some background material and preliminary results, we consider the case where $X$ is the set of shortest vectors
of one of the exceptional lattices $E_6$, $E_7$, $E_8$, $\Lambda_\m$ (the Leech lattice) and determine for each:
(i) the smallest degree of a non-trivial polynomial in the ideal, and (ii) the smallest $k$ for which the ideal admits
a generating set of polynomials all of degree $k$ or less. As it turns out, in all four cases mentioned above, these
two values coincide, as they also do for the icosahedron, our introductory example. The paper concludes with a
discussion of these two parameters, two open problems regarding their equality, and a
few remarks concerning connections to cometric association schemes.
\end{abstract}

%%%%%%%%%%%%%%%%%%%%%%%%%%%%%%%%%%%%%%%%%%%%%%%%%%%%%%%%%
\section{Introduction}

The Leech lattice $\Lambda_\m$ is well-studied in several mathematical contexts (see \cite[Section~4.11]{splag} and the references therein). 
This remarkable unimodular lattice in $\re^\m$ plays a fundamental role in the study of exceptional finite simple 
groups, gives an optimal sphere packing,  and its shortest vectors encode an optimal kissing configuration in 
$\re^\m$. There is an intriguing connection to modular forms and number theory  as well.
Among other examples, we consider here the 196560 shortest vectors of $\Lambda_\m$ and describe the ideal of polynomials
in 24 variables with real (or complex) coefficients which vanish on these points. 
We find that, while every polynomial in this ideal of total degree less than six is divisible by the equation of the 
sphere containing these 196560 points, the polynomials of degree six in the ideal generate the full ideal.

After looking at an instructive example and introducing the basic machinery we will need to study the shortest 
vectors of the Leech lattice, we first apply the techniques to three smaller examples of spherical codes coming 
from lattices. It is well-known that the shortest vectors of the $E_8$ lattice give us a spherical code of size 
$240$ (and again an optimal kissing configuration) in $\re^8$ and  the shortest
vectors of lattices $E_7$ and $E_6$ may be obtained from this configuration by intersecting it with  affine subspaces of 
codimension one and two, respectively \cite[Section~4.8]{splag}. In each case, we determine ``nice'' generating 
sets for the ideals of these configurations and find that the smallest degree of a non-trivial polynomial (polynomial multiples 
of the equation of the sphere itself being viewed as ``trivial'') is equal to the maximum degree of a polynomial
in these nice generating sets. This prompts us to ask when these two parameters are equal. We explore this
question at the end of this paper where we prove some simple inequalities on these two parameters and 
tie this material in to the study of cometric association schemes. These connections suggest
a rich interplay between the theory of (representations of) association schemes and elementary % real
algebraic geometry.

Many of the computations performed for this project were executed using the {\sc maple} and 
{\sc singular} computer algebra systems.

%%%%%%%%%%%%%%%%%%%%%%%%%%%%%%%%%%%%%%%%%%%%%%%%%%%%%%%%%
\subsection{A simple example}
\label{sec:icos}

As a way of introducing some of the terminology and techniques involved, we begin by considering the icosahedron, 
which we view as a set $X$ of 12 unit vectors in $\re^3$, closed under multiplication by $-1$ and
having pairwise inner products $\pm1, \pm 1/\sqrt{5}$.  For $a\in X$, the zonal polynomial
$$ Z_{f,a}(Y_1,Y_2,Y_3) = \left( a \cdot Y - 1 \right)  \left( a \cdot Y - 1/\sqrt{5} \right)  \left( a \cdot Y + 1/\sqrt{5} \right)  \left( a \cdot Y + 1 \right), $$
(which can be written $Z_{f,a}(Y) = f( a \cdot Y)$ for $f(t) = (t^2-1)(t^2 - 1/5)$) clearly vanishes at all points of $X$. We would like an
efficient (or simple, at least) description of all polynomials in three variables which have every element of $X$ as a zero; this 
set of polynomials forms  an ideal in $\cx[Y_1,Y_2,Y_3]$ which is denoted by $\cI(X)$. As with any spherical code, the quadratic
$$ \Nm := \sum_i Y_i^2 - 1 $$
vanishes on the entire unit sphere, so vanishes on $X$. Any polynomial divisible by $\Nm$ has this property as well and all these polynomials
are considered {\em trivial} for our purposes.

The icosahedron is a well-known spherical $5$-design \cite{dgs}. So \pref{tdes} below implies that the ideal 
$\I = \cI(X)$ contains no non-trivial polynomial of degree two or less. 
 But since $X$ is an antipodal code (i.e., closed under multiplication by $-1$), we may modify 
the degree 4 zonal polynomial given above to obtain several degree 3 polynomials in the ideal. Let $a\in X$ and let $b\in \re^3$ be any
nonzero vector orthogonal to $a$. Then 
$$F(Y) = (b\cdot Y)( a\cdot Y - 1/\sqrt{5}) ( a\cdot Y + 1/\sqrt{5}) $$
belongs to $\I$ as $F( c ) = 0$ for all $c\in X$. (Either $a\cdot c = \pm 1/\sqrt{5}$ or $c=\pm a$ and  $c$ is orthogonal to $b$.)
For  fixed $a$, one may make two linearly independent choices for $b$, thereby obtaining a pair of polynomials whose common
zero set consists of the two planes $\{ c \mid a\cdot c =  \pm 1/\sqrt{5}  \}$ 
together with the line joining $a$ to $-a$.

One straightforwardly employs software such as {\sc maple} to verify that $\Nm$ together with 12 polynomials 
$F$ obtained as above (two choices of $b$ for
each antipodal pair $\{a,-a\}$) generate a radical ideal with exactly $X$ as its zero set. Thus we have determined that $\I = \cI(X)$ is 
generated by a  set of polynomials, each of degree at most three and, yet, the smallest degree of any nontrivial polynomial in $\I$ is
also three.

%%%%%%%%%%%%%%%%%%%%%%%%%%%%%%%%%%%%%%%%%%%%%%%%%%%%%%%%%
\section{Preparatory lemmas}
\label{sec:prep}

Consider the algebra of polynomials $\re[Y] = \re[Y_1,\ldots, Y_m]$ and, as necessary, its extension 
$\cx[Y_1,\ldots, Y_m]$.  For a fixed finite $X \subseteq \re^m$ of size $v$,  we have  the natural 
evaluation map from this algebra to the direct product $\re\times \re \times \cdots \times \re$ of $v$ copies of $\re$
\begin{equation}
\label{eqn:eps}
 \varepsilon : \re[Y] \rightarrow \re^v \qquad \text{ given by} \qquad \varepsilon: F\mapsto (F(a),F(b),F( c), \ldots  )
\end{equation}
where $X = \{ a,b,c, \ldots \}$ (cf. \cite[Prop.~2.6]{fulton}).  We are interested in the ideal of all polynomials which map to the zero vector under
$\varepsilon$; this is the kernel of the algebra homomorphism $\varepsilon$.  
Our goal is to find combinatorially meaningful generating sets for this ideal and use these descriptions in our study of association schemes.
(We defer our discussion of the connection between this homomorphism and cometric association schemes
to a follow-up paper \cite{billsidealpaper}.)  If we are content with just any generating set, our task seems trivial. For example, when $|X|=1$,
our ideal $\ker \varepsilon$ takes the form
$$ \langle Y_1 - a_1, \ldots, Y_m - a_m \rangle$$
where $X = \{ (a_1,\ldots, a_m) \}$ and for larger finite $X$, our ideal can be expressed as the intersection of ideals of this form. (In
Section 4.3 of \cite{iva}, compare Theorem 7, Theorem 15 and Proposition 16.)

We use the following notation for the standard operations of algebraic geometry. (See, e.g., \cite[Chapter~15]{df} for a basic introduction.)
For a set $X\subseteq \cx^m$, we let $\cI(X)$ denote the ideal of all polynomials in $\cx[Y_1,\ldots, Y_m]$ that vanish at each point in $X$.
And if $\cS$ is any set of polynomials in $\cx[Y_1,\ldots, Y_m]$, we denote by $\cZ(\cS)$ the zero set of $\cS$, the collection of all points
$a$ in $\cx^m$ which satisfy $F(a)=0$ for all $F\in \cS$. Note that, when $X$ is finite, we have $\cZ(\cI(X))=X$ and, by the Nullstellensatz
(see, e.g., \cite[p21]{fulton}, \cite[p173]{iva}),  $\cI(\cZ(\J))=\Rad(\J)$, where $\Rad(\J)$ denotes the {\em radical} of ideal $\J$,  the
ideal of all polynomials $F$ such that $F^n\in \J$ for some positive integer $n$.

\bigskip

Let $f(t)$ be a polynomial in the variable $t$ and let $a \in \re^m$. Then, with $a\cdot Y := a_1Y_1 + \cdots + a_m Y_m$, the {\em
zonal polynomial} determined by $f$ and based at $a$ is the polynomial $Z_{f,a}(Y) \in \re[Y_1,\ldots Y_m]$ given by
$$ Z_{f,a}(Y) = f( a\cdot Y). $$
For example, if $f(t) = \prod_{h=0}^d (t-\omega_h)$, then 
$$ Z_{f,a}(Y) = \left(  a_1Y_1 + \cdots + a_m Y_m - \omega_0 \right) \cdots  \left(  a_1Y_1 + \cdots + a_m Y_m - \omega_d \right)~.$$
Observe that, if  $f(t)$ is defined as in the previous sentence and  $\{ a\cdot b \mid b\in X \} \subseteq \{\omega_0,\ldots, \omega_d\}$,
then  $Z_{f,a} \in \cI(X)$. 

For a finite subset $X \subseteq \re^m$ with inner product set 
$$  \left\{ a\cdot b \mid a,b \in X \right\} = \{ \omega_0, \ldots, \omega_d\}$$
the {\em zonal ideal} of $X$ is the ideal 
$$  \Z_X  := \langle Z_{f,a}(Y) \mid a \in X \rangle $$
where $f(t)  = \prod_{h=0}^d (t-\omega_h)$. Clearly $\Z_X \subseteq \cI(X)$, but the two are not always equal. This occurs,
for example, when we take $X$ to consist of all but one or two of the vertices of the icosahedron. A more interesting 
example where equality fails is when $X$ is the 4-cube in $\re^4$ in which case $\cZ(\Z_X)$ contains all vertices of the 24-cell; this phenomenon 
is repeated with any other non-complete set of real mutually unbiased bases \cite{lmo}.

\begin{lemma} 
\label{lem:realzeros}
The following statements hold true for complex zeros of zonal polynomials:
\begin{itemize}
\item[(i)] Suppose $a$ is a nonzero vector in $\re^m$ and $\omega_0,\ldots,\omega_d$ are real numbers and consider the
polynomial $F \in \re[Y_1,\ldots, Y_m]$ defined by 
$$ F(Y) = (a\cdot Y - \omega_0) \cdots  (a\cdot Y - \omega_d) ~ .$$
If $z = r+\i s \in \cx^m$ is any zero of $F$ with $r,s \in \re^m$, then $a\cdot s =0$.  
\item[(ii)] If $\I$ is an ideal
generated by polynomials of this form, say $\I = \langle F_1, \ldots, F_t \rangle$ with 
$$ F_h(Y) = \prod_{j=0}^{d_h} \left( a^{(h)} \cdot Y - \omega_{h,j} \right) $$
with each $d_h \ge 0$, each $a^{(h)} \in \re^m$ and all $\omega_{h,j} \in \re$, then any common zero $z\in \cZ(\I)$,
$z = r+\i s$, has imaginary part $s$ orthogonal to all $a^{(h)}$ ($1\le h \le t$).
\item[(iii)] In particular, if the vectors $a^{(h)}$ in part {\sl (ii)} span
$\re^m$, then every zero of ideal $\I$ is real.
\end{itemize}
\end{lemma}

\begin{proof} We prove only part {\sl (ii)}. Write $z=r+\i s$ with $r,s\in \re^m$. If $F_h(z) = 0$, then there exists some $1\le j_h \le d_h$ with 
$a^{(h)} \cdot z = \omega_{h,j_h}$ so that $a^{(h)} \cdot r = \omega_{h,j_h}$ and $a^{(h)} \cdot s = 0$. 
Thus the real vector $s$ lies in the nullspace of the $t\times m$ matrix
whose rows are $a^{(h)}$, $1\le h \le t$. \qed
\end{proof}

\begin{corollary}
\label{cor:zonal}
Let $X$ be any spanning subset of $\re^m$. Then the zonal ideal $\Z_X$ of $X$ has only real zeros, as does
any ideal that contains it. \qed
\end{corollary}

%%%%%%%%%%%%%%%%%%%%%%%%%%%%%%%%%%%%%%%%%%%%%%%%%%%%%%%%%
\subsection{The antipodal case}
\label{sec:antip}

A spherical code $X$ is said to be {\em antipodal} \cite[Example~5.7]{dgs} if $-X=X$, i.e., for every $a\in X$, the point $-a$ belongs
to $X$ as well. Suppose $X$ is an antipodal spherical code in $\re^m$ with inner product set $\{ \omega_0, \ldots, \omega_d\}$ where, without loss of 
generality, $\omega_0=1$ and $\omega_d=-1$. Then as with our introductory example --- the icosahedron --- the ideal $\cI(X)$ 
not only contains all the zonal polynomials $Z_{f,a}(Y)$ for $a\in X$ and $f(t) = \prod (t-\omega_i)$, but also the ``sliced
zonal polynomials''
\begin{equation}
\label{eqn:sliced-def}
S_{f,a,b}(Y) :=  \left( a \cdot Y - \omega_1 \right) \cdots \left( a \cdot Y - \omega_{d-1} \right) (b\cdot Y) 
\end{equation}
where $a\in X$ and  $b$ is any nonzero vector in $\re^m$ orthogonal to $a$. Note that this polynomial has degree $d$, one less than the degree of 
the zonal polynomial. As one of our goals here is to find generators of small degree, we will typically prefer these sliced zonals over
the zonal polynomials. 

Let us now describe the ideal  $\S_X$ generated by all sliced zonal polynomials.
Suppose that $X$ is an antipodal spherical code in $\re^m$ with inner product set 
$$1=\omega_0 > \omega_1 > \cdots > \omega_{d-1} > \omega_d = -1$$ 
and, for each antipodal pair $\pm a$ choose a set of $m-1$ linearly independent vectors $b_{a,i}$ ($1\le i \le m-1$) 
all orthogonal to $a$. (E.g., we may take $\{a, b_{a,1}, \ldots, b_{a,m-1} \}$ to be an orthogonal basis for $\re^m$.) 
For each such pair $\pm a$, consider the set of polynomials
\begin{equation}
\label{eqn:Ba}
 B_a = \left\{  \left( b_{a,i} \cdot Y \right) \prod_{h=1}^{d-1} \left( a \cdot Y - \omega_h \right) \mid 1 \le i \le m-1 \right\} ~. 
 \end{equation}
The {\em ideal of sliced zonals} for this set $X$ is the ideal
$$ \S_X = \left\langle \bigcup_{a\in X} B_a \right\rangle $$
generated by the union of all the sets $B_a$ as $\pm a$ ranges over the antipodal pairs in $X$. (Note  that the ideal
generated by $B_a$, and hence $\S_X$ itself, is independent of the choice of the vectors $b_{a,i}$ as any other sliced zonal polynomial for this pair $\pm a$ is
a linear combination of those chosen.)

\begin{corollary}
\label{cor:slicedzonal}
Let $X$ be an antipodal spherical code in $\re^m$ such that $X$ spans $\re^m$. Assume $d\ge 3$. 
Then the ideal of sliced zonal polynomials $\S_X$ defined above has only real zeros, as does any ideal that contains it.
\end{corollary}

\begin{proof}
The proof is similar to the proof of \cref{zonal}. But in this case, when $z=r+\i s$ with $r,s\in \re^m$  is a common zero of every sliced zonal polynomial,
we may only infer that, for each $a\in X$, either $s \bot a$ or $s$ is parallel to $a$. So $s=0$ is  forced  unless any 
pair of elements from $X$ are either parallel or orthogonal. This can only happen if $d\le 2$.   \qed
\end{proof}

\begin{remark}
The orthoplex $X = \{ \pm e_i \mid 1\le i\le m\}$, consisting of the standard basis vectors together with their antipodes, is the only
antipodal spherical code with $d=2$; the claim of real zeros fails only in this case.
\end{remark}

%%%%%%%%%%%%%%%%%%%%%%%%%%%%%%%%%%%%%%%%%%%%%%%%%%%%%%%%%
\subsection{Derived designs}

A subset $X \subseteq \Omega_m$ of the unit sphere in $\re^m$ is a {\em spherical design} of strength $t$ 
(or a {\em spherical $t$-design}) if
the average over $X$ of any polynomial $F(Y)$ of total degree at most $t$ in $m$ variables is exactly equal to 
its average over the entire sphere $\Omega_m$.  These objects are well-studied; see \cite{dgs}, \cite[p.~xxii]{banito},
\cite[Chapter 14,15]{godsil} for background and \cite{banban} for recent developments. 
The following result is a useful observation of Bannai (but also see papers of M\"{o}ller, e.g., \cite{moller}):

\begin{proposition}
\label{pro:tdes}
Let $X$ be a spherical design of strength $t$ and suppose $F \in \cI(X)$. If $\deg F \le t/2$, then $F$ is trivial; i.e., $\Nm$ divides $F$.
\end{proposition}

\begin{proof} The polynomial
$F^2$ is nonnegative on the sphere but vanishes on $X$. Since $X$ is a spherical $t$-design, the integral of $F^2$
over the sphere is exactly zero,  so $F$ itself must vanish on the entire sphere. \qed
\end{proof}

\noindent {\bf Notation:} Henceforth, we will not always bother to scale our vectors to unit length. If $X$ is a finite subset of 
$\re^m$ with $a\cdot a$ constant over all $a\in X$ (i.e., $X$ is a ``spherical configuration''),  we will use $\Nm$ to denote 
the polynomial $\sum_i Y_i^2 - a\cdot a$ and will consider any multiple of $\Nm$ a ``trivial'' polynomial.

\medskip

Let $X$ be a spherical configuration in $\re^m$ and consider an affine subspace $T$, of dimension $k$ say,  
which contains at least two points of $X$. Then 
$T$ intersects this  sphere in a sphere and $T$ may be coordinatized so that this sphere has unit radius and is centered at the origin
in a vector space $\re^{k}$ and $X' := X \cap T$ corresponds naturally to a spherical code in $k$-dimensional space. Special cases include  {\em derived
codes} and {\em derived designs}  (see \cite{dgs}).  Our goal in this section is to relate the ideal of $X'$ to the ideal determined
by $X$.

It will be convenient to view  $m$-dimensional  Euclidean space $\cE$ as having two coordinatizations: the natural vector space structure $\re^m$
and another vector space structure $V$ emerging from our coordinatization of $T$. For a point $P$ in the Euclidean space  $\cE$,
we will let  $p$ and $p'$
denote its $\re^m$- and $V$-coordinates, respectively. We will also use $\Omega_m$ to denote the sphere containing $X$
in $\re^m$ and $\Omega_k$ to denote the unit sphere, which contains $X'$, in the subspace $T$ of $V$.

For a $k$-dimensional subspace $T$ of $\cE$ intersecting $\Omega_m$ nontrivially, let $\mathcal{O}$ denote the center of the sphere
$T \cap \Omega_m$ and, with $\mathcal{O}$ as origin, endow $\cE$ with a vector space structure $V$ so that, in coordinates
$Y'_1, Y_2', \ldots, Y'_m$ for $V$, subspace $T$ has equation $Y'_{k+1}=\cdots = Y'_m = 0$.  To distinguish the vector space structure $\re^m$ from
that of $V$ we will use indeterminates $Y_1,\ldots, Y_m$ when working in $\re^m$.  These coordinate systems are related by an affine 
change-of-coordinate system
$$  Y = C Y' + d $$
with invertible Jacobian $C$.

\bigskip

Now for a given spherical code $X\subset \Omega_m$ in $\re^m$, let $\F = \{F_1,\ldots, F_n\}$ be a generating set for $\I = \cI(X)$. The Jacobian
matrix with respect to $\F$ and $Y_1,\ldots, Y_m$ at a point $p$ is given by 
$$ \Jac(\F,p) = \left[ \begin{array}{ccc}
\frac{ \partial F_1 }{ \partial Y_1} \big|_p  & \cdots  & \frac{ \partial F_1 }{ \partial Y_m} \big|_p  \\
     \vdots   &   \ddots &   \vdots \\
\frac{ \partial F_n }{ \partial Y_1} \big|_p & \cdots  & \frac{ \partial F_n }{ \partial Y_m}  \big|_p 
\end{array} \right] ~. $$
On the other hand, consider the $n$ polynomials $\F' = \{F_1',\ldots, F'_n\}$  in the variables $Y_i'$ given by 
$$ F'_i(Y') = F_i( CY' + d)~.$$
Clearly, the $F_i$ are recovered from the $F'_i$ by  $F_i(Y) = F'_i( C^{-1}Y - C^{-1}d)$. 
Since we consider $X'$ as an algebraic set in $T$ itself,
we also need 
$$ G_i(Y'_1,\ldots, Y'_k) = F'_i( Y'_1,\ldots, Y'_k, 0,\ldots, 0)$$
and we observe $ \frac{ \partial G_i}{\partial Y'_j} = \frac{ \partial F'_i }{ \partial Y'_j}$ when $j\le k$.  We now study the relationship
between $X'$ and the ideal $\J = \langle \G \rangle$ in $k$ variables, where $\G = \{ G_1,\ldots, G_n \}$.

\begin{proposition}  
\label{pro:derived}
With notation as above, we have
\begin{itemize}
\item[(i)] If $\cZ(\I)=X$ in $\re^m$, then $\cZ(\J) = X'$ in $T$;
\item[(ii)] If $p \in X'$ is a simple zero of $\I$, then $p'$ is a simple zero of $\J$; 
\item[(iii)] If $\I = \cI(X)$, then $\J = \cI(X')$ in $\cx[Y_1',\ldots, Y_k']$.
\end{itemize}
\end{proposition}

\begin{proof}
As above, we have $\I = \ip{\F} = \langle F_1,\ldots, F_n \rangle$ and $\J = \ip{\G} = \ip{G_1,\ldots, G_n}$. Let $P \in X'$ with $\re^m$-coordinates $p$ 
and $V$-coordinates $p'$. Choose $G_i$ ($1\le i \le n$) and observe
$ G_i( p'_1,\ldots, p'_k ) = F_i' (p') = F_i( p ) = 0$ since $P \in X$.
Therefore $X' \subseteq \cZ(\J)$. Conversely, suppose $P \in \cZ(\J) \subset T$. Then, for each $i$,  we have $G_i( p'_1,\ldots, p'_k ) = 0$
so that $F'_i(p') = 0$ and $F_i( p ) = 0$ so that $P \in X$. But $P \in T$ so $P \in X'$. This proves {\sl (i)}.

By the Chain Rule, the Jacobians of the systems $\{F_i\}$ and $\{F'_i\}$ of polynomials in their respective coordinate systems
are related by 
$$ \Jac(\F',p') = \Jac(\F,p) \cdot C $$
where $C$ is the matrix of coefficients of the affine change-of-coordinates defined above. Since $C$ is invertible, $ \Jac(\F',p')$ and $\Jac(\F,p)$
have the same column rank. So, if $p$ is a simple zero (i.e., a smooth point) of $\I$, we have that $\rank \Jac(\F,p) = m$ so that $\rank \Jac(\F',p') = m$ also.
Now $\Jac(\G,[p'_1,\ldots,p'_k])$ is just a submatrix of $ \Jac(\F',p')$, obtained by restricting to the first $k$ columns. So,  if $p$ is a simple zero of $\I$, 
then this latter Jacobian has full column rank and the Zariski tangent space of $\J$ at $p'$ is zero-dimensional, proving {\sl (ii)}.

To prove {\sl (iii)}, we use {\sl (ii)} to obtain
$$|X'| = \sum_{P \in X'} \mult ( P ) = \dim \re[Y']/ \J  \ge \dim \re[Y']/ \Rad(\J) = |X'| $$
(see Section 5.3, Proposition 8 in \cite{iva}) which implies that $\dim \re[Y']/ \J  =\dim \re[Y']/ \Rad(\J) $ so that $\J = \Rad(\J)$ and $\J$ is a radical ideal.
Now, applying this, the Nullstellensatz,  and {\sl (i)} in turn, we have
$$ \J  = \Rad(\J) = \cI( \cZ(\J) ) =\cI(X') $$
proving {\sl (iii)}. \qed
\end{proof}

%%%%%%%%%%%%%%%%%%%%%%%%%%%%%%%%%%%%%%%%%%%%%%%%%%%%%%%%%
\section{Ideals for $E_6$, $E_7$ and $E_8$}
\label{sec:e8}

%In this paper, we adopt a perhaps unusual notational convention. 
First, a remark on notation. In this paper, we denote by $E_6$,
$E_7$, $E_8$ and $\Lambda_{24}$ the four famous lattices we consider and use $\E_6$,
$\E_7$, $\E_8$ and $\L_{24}$ to denote the set of shortest vectors of each of these
lattices, respectively.

Following \cite[Eq.~(97), p120]{splag}, we consider the spherical configuration (on a sphere of radius $\sqrt{2}$ centered at the
origin in $\re^8$) consisting of the shortest vectors of the $E_8$ lattice
$$ \E_8 = 
\left\{ e_i \pm e_j \mid 1 \le i\neq j \le 8\right\} \cup \left\{ \left(\pm \frac{1}{2}, \pm \frac{1}{2},\ldots,  \pm \frac{1}{2} \right) \bigm| \text{an even number of $-$ signs} \right\}. $$

For each pair $\pm a$ of antipodal  points in $X$, choose a set $B_a$  as in Equation (\ref{eqn:Ba}) of seven sliced zonal polynomials of degree
four. Now let  $\G$ be the set consisting of $\Nm$ along with these $120 \cdot 7$ sliced zonal polynomials.

\begin{theorem} 
\label{thm:E8}
The ideal $\I = \langle \G \rangle$ generated by the set $\G$ defined in the preceding paragraph has the following properties:
\begin{itemize}
\item[(i)] $\cZ(\I) = \E_8$ in $\cx^8$; 
\item[(ii)]  each zero of $\I$ is simple, so $\I$ is a radical ideal;
\item[(iii)]  the smallest degree of a non-trivial polynomial in $\cI(\E_8)$ is four;
\item[(iv)]  $\cI(\E_8)$ is generated by a set of polynomials all having degree four or less.
\end{itemize}
\end{theorem}

\begin{proof}
The only technical step is {\sl (ii)}, and an almost identical argument is given for the Leech lattice example in the proof of 
\tref{Leech}, so we give only a sketch here.

We apply \cref{slicedzonal} to see that the ideal $\I$, which contains the sliced zonal 
ideal $\S_X$, has all of its zeros contained in $\re^8$. Since $\Nm \in \I$, every zero 
lies on the sphere of radius $\sqrt{2}$ centered at the origin.  If $a\in \cZ(\E_8)$, then
$a \cdot b \in \{-2,-1,0,1,2\}$ for every $b \in \E_8$. It then follows that $a$ has
integer inner product with every vector in the lattice $E_8$ spanned by the vectors
in $\E_8$. But $E_8$ is a unimodular lattice, so this implies $a \in E_8$ and, given
that $\Nm \in \G$, we must have $a\in \E_8$. 
This proves {\sl (i)}. 
(Alternatively, we could use the optimality of $\E_8$ as a kissing configuration \cite[p120]{splag}, but this
argument, suggested by the referee, is more elegant.)

Now that we have identified all of the zeros of $\I$, we prove
that each of these is a simple zero  
by proving that the Zariski tangent space at the point is zero-dimensional. The lack of 
multiple zeros then implies that the ideal is indeed radical, giving us {\sl (ii)}. In order to compute the dimension of the tangent space at a point
$a$, we need to locate eight polynomials in $\I$ whose gradients span $\re^8$. The gradient of a sliced zonal based at $c\in X$,
evaluated at $a \neq \pm c$ is a nonzero scalar multiple of $c$ itself, so we have sufficient supply of such polynomials. (See
the proof of \tref{Leech} for full details.)

Since our generating
set contains non-trivial polynomials of degree four and $\E_8$ is a spherical 7-design,  \pref{tdes} gives us {\sl (iii)}.
Finally, once the proof of {\sl (ii)} is complete, we apply it, together with {\sl (i)} and the Nullstellensatz, 
to conclude that our generating set indeed generates $\cI(\E_8)$. \qed 
\end{proof}
The ideal $\cI(\E_8)$ contains a number of other interesting polynomials of degree four. For example, if 
$\{\{ i_1,i_2,i_3,i_4\}, \{ j_1,j_2,j_3,j_4\}\}$ is any partition of $\{1,\ldots,8\}$ into two sets of size four,
then the polynomial $Y_{i_1} Y_{i_2} Y_{i_3} Y_{i_4} -  Y_{j_1}  Y_{j_2}  Y_{j_3}  Y_{j_4} $ is easily seen to
vanish on each of the 240 vectors given above. 
(But these polynomials also vanish at any point with at least five zero entries.)

\bigskip

We recall that the shortest vectors of the $E_7$ root lattice are conveniently given as a derived design of 
$\E_8$ (see \cite[p120]{splag}):
$$ \E_7 = \left\{ a \in \E_8 \mid a_7 = a_8 \right\} ~. $$
This allows us to identify a generating set for the ideal $\cI(\E_7)$. We may immediately apply  \pref{derived} 
to obtain a generating set for $\cI(\E_7)$ consisting of polynomials of degree two and four.
However, we can do better by using just a subset of these polynomials and making a small modification.

We have chosen $\E_7$ to be the set of 
all vectors in $\E_8$ having inner product zero with $(0,0,0,0,0,0,1,-1)$. There are 56 vectors in $\E_8$ having inner 
product one with this vector. For each of these 56 vectors $b$, consider the zonal polynomial
$$  C_b(Y) := (b\cdot Y - 1)(b\cdot Y) (b \cdot Y + 1) ~. $$
Clearly since neither $b$ nor $-b$ belongs to $\E_7$, this gives a cubic polynomial which vanishes on each point of $\E_7$.
The ideal $\I = \langle C_b(Y) \mid b \in \E_8, \ b_7-b_8=1 \rangle$ generated by these 56 cubics, together with  $\Nm$ is 
easily verified by computer to be a radical ideal with $\cZ(\I) = \E_7$. But we can prove this directly as well.

\begin{theorem} 
\label{thm:E7}
Let $\I = \langle C_b(Y) \mid b \in \E_8, \ b_7-b_8=1 \rangle$ as constructed in the previous paragraph. Then 
\begin{itemize}
\item[(i)] $\cZ(\I) = \E_7$;
\item[(ii)] $\I$ is a radical ideal;
\item[(iii)]  the smallest degree of a non-trivial polynomial in $\cI(\E_7)$ is three;
\item[(iv)]  $\cI(\E_7)$ is generated by a set of polynomials all having degree three or less. \qed
\end{itemize}
\end{theorem}

\begin{proof}
Let $\mathcal{H}$ be the set of $56$ polynomials on the subspace $T$ of vectors in $\re^8$ orthogonal to $a=(0,0,0,0,0,0,1,-1)$
given by 
$$ \left\{ (b\cdot Y - 1)(b\cdot Y) (b \cdot Y + 1)  \mid b \in \E_8, \ b\cdot a = 1 \right\} ~.$$
It is messy, but not hard, to verify that each zonal polynomial $Z_{f,a}$ for $a\in \E_7$, $f(t)=t(t^2-1)(t^2-4)$, belongs to 
the ideal generated by these 56 polynomials. (A sample computation is included below.) 
\cref{zonal} tells us that any common zero of these 56 polynomials
has only real entries. So, including $\Nm$ in our generating set, we find that any vector in $\cZ(\I)$ lies on the unit sphere
in $\re^8$ and is either equal to or at least $\pi/3$ radians away from any element of $\E_8$. Again applying the optimality of 
$\E_8$ as a kissing configuration, we deduce that $\cZ(\I) = \E_7$ inside subspace $T$. We may 
then use the technique outlined in the previous proof to verify that each element of $\E_7$ 
is a simple zero of $\I$; so this ideal is indeed radical.  Since $\E_7$ is a spherical 5-design, the generating set we have 
described gives us both {\sl (iii)} and {\sl (iv)}. \qed
\end{proof}

\begin{example}
Let $c=(1,1,0,0,0,0,0,0)$. We express the degree five zonal polynomial 
$$Z_{f,c}(Y) = (c\cdot Y - 2) (c\cdot Y - 1) (c\cdot Y ) (c\cdot Y + 1) (c\cdot Y + 2) $$
(restricted to the subspace $Y_7=Y_8$) in terms of the four cubics $C_b(Y)$ as $b$ ranges over
$$ b_1,b_3=(\pm 1,0,0,0,0,0,1,0), \qquad \text{and} \qquad 
b_2,b_4=(0,\pm1,0,0,0,0,1,0).
$$
Let 
\begin{eqnarray*}
Q_1(Y) =  \frac{1}{2} (Y_1+Y_2- 4Y_7)( \phantom{-}  Y_1+4Y_2 + Y_7), &\qquad& R_1(Y) = 3Y_2^2 + 5Y_7^2 - 2, \\
Q_2(Y) =  \frac{1}{2} (Y_1+Y_2- 4Y_7)(\phantom{-}  4Y_1+Y_2 + Y_7),     &\qquad&R_2(Y) = 3Y_1^2 + 5Y_7^2 - 2,   \\ 
Q_3(Y) =  \frac{1}{2}  (Y_1+Y_2+ 4Y_7)(-Y_1-4Y_2 + Y_7),  &\qquad& \\
Q_4(Y) =  \frac{1}{2}  (Y_1+Y_2+ 4Y_7)(-4Y_1-Y_2 +Y_7). &\qquad& 
\end{eqnarray*}
Then we have 
$$  Z_{f,c} =  \left[Q_1 + R_1 \right] C_{b_1} + 
 \left[Q_2  + R_2 \right] C_{b_2}  +   \left[Q_3 - R_1 \right] C_{b_3} +   \left[Q_4 - R_2 \right] C_{b_4} ~  .$$
Likewise, each $Z_{f,c}(Y)$ for $c \in \E_7$ is shown to belong to the ideal $\cI$ in \tref{E7} by finding
similar (but more complicated) expressions.
\end{example}

\bigskip

Next, the shortest vectors of the root lattice $E_6$ may be taken as a derived spherical design of $\E_7$. As the parameters
we compute (smallest degree of a nontrivial polynomial in the ideal, smallest possible maximum degree of polynomials in any 
generating set) are independent of the choice of any full-dimensional representation of the configuration (up to invertible 
affine transformation), we find it convenient again to work with $\E_6$ given as a derived design of $\E_8$ (see \cite[p120]{splag}):
$$ \E_6 = \left\{ a \in \E_8 \mid a_6 = a_7 = a_8 \right\} ~. $$
This, too, is a spherical 5-design. Now we apply \tref{E7} together with \pref{derived} and
 \pref{tdes} to obtain

\begin{theorem} 
\label{thm:E6}
Let $\I$ be the ideal considered in \tref{E7} above and let $\J$ be the ideal
in $\re[Y'_1,\ldots, Y'_6]$ as in \pref{derived}. Then
\begin{itemize}
\item[(i)] $\cZ(\J) = \E_6$;
\item[(ii)] $\J$ is a radical ideal;
\item[(iii)]  the smallest degree of a non-trivial polynomial in $\cI(\E_6)$ is three;
\item[(iv)]  $\cI(\E_6)$ is generated by a set of polynomials all having degree three or less. \qed
\end{itemize}
\end{theorem}

%%%%%%%%%%%%%%%%%%%%%%%%%%%%%%%%%%%%%%%%%%%%%%%%%%%%%%%%%
\section{The Leech lattice}
\label{sec:leech}

The shortest vectors of the Leech lattice are described by Leech in \cite[Sec.~2.31]{leech2} (see also \cite{leech1}) 
but can also be found in more recent references
such as \cite[p133]{splag}.  These $196560$ vectors (scaled by $\sqrt{8}$ for convenience) 
are described in the following table; we will denote this set of vectors by $\L_\m$.

\begin{center}
\begin{tabular}{|c|r|l|}\hline
Shape & Number & Description \\ \hline\hline
$(\pm2^8, 0^{16} )$    &$2^7 \cdot 759$ & same support as some weight-8 Golay \\
                                    &                          & codeword, even number of minus signs \\ \hline
$(\mp 3, \pm1^{23})$  & $24 \cdot 2^{12}$ &  upper signs appear on the support of some \\
                                    &                              &  Golay codeword \\ \hline
$(\pm 4^2, 0^{22})$ &  $4 \cdot \binom{24}{2} $  & any two non-zero entries of absolute value 4 \\ \hline 
\end{tabular}
\end{center}
Leech vectors of the first two types are specified in terms of the supports of codewords of the extended binary Golay code
$\mathsf{G}_\m$. The reader unfamiliar with this code may refer to any standard text on coding theory, or any of \cite{del,splag,bcn}.
For the sake of completeness, one may instead take $\mathsf{G}_\m$ to be the binary rowspace of the $12\times 24$ matrix 
$[I | J-A]$ where $A$ is the adjacency matrix of the icosahedron and $J$ is the all-ones matrix.

One routinely checks (or verifies through the references) that, for any $a,b\in \L_\m$, the inner product $a\cdot b$ belongs to the set
$\{32,16,8,0,-8,-16,-32\}$ so define $\omega_i = 24-8i$ for $1\le i\le 5$ and $\omega_0,\omega_6 = \pm 32$.

\bigskip

For  $a \in \L_\m$, consider the zonal polynomial based at $a$ and determined by $f(t) = \prod_{h=0}^6 (t-\omega_h)$:
$$ Z_{f,a}(Y) = \prod_{h=0}^6 \left( a\cdot Y - \omega_h\right) = (a_1 Y_1 + \cdots a_\m Y_\m - 32) \cdots (a_1 Y_1 + \cdots a_\m Y_\m + 32) ~.$$
Next, for $a\in X$ and nonzero $b \in \re^\m$ with $b\cdot a =  0$, construct the degree six sliced zonal polynomial (\ref{eqn:sliced-def})
$$ S_{f,a,b}(Y) = (b \cdot Y)  \prod_{h=1}^5 \left( a\cdot Y - \omega_h\right) ~. $$
As stated in \sref{antip}, $S_{f,a,b}$ vanishes on $X$ whenever $a\in X$ and $b\bot a$.

Now we construct our ideal for $\L_\m$. For each antipodal pair $\pm a$ of vectors in $\L_\m$, extend $\{a \}$ to an
orthogonal basis $\{ a, c^{(1)}(a), \ldots, c^{(23)}(a) \}$ for $\re^\m$. 
It will be convenient to abbreviate $S_{f,a,c^{(i)}(a)}$ to $S_{a,i}$ for the remainder of this discussion. Now  consider the set 
$$ \G  = \left\{ S_{a,i}  \mid \pm a\in X, \ 1\le i\le 23 \right\} \cup \left\{  \Nm \right\}~.$$  
We claim that $\I = \langle \G \rangle $ is the ideal we seek.

\begin{theorem}
\label{thm:Leech}
Let $\L_\m$ be the set of shortest vectors of the Leech lattice, as defined above. Let $\I = \langle \G \rangle$. Then
\begin{itemize}
\item[(i)] $\cZ(\I) = \L_\m$ in $\cx^\m$; 
\item[(ii)]  each zero of $\I$ is simple, so $\I$ is a radical ideal;
\item[(iii)]  the smallest degree of a non-trivial polynomial in $\cI(\L_\m)$ is six;
\item[(iv)]  $\cI(\L_\m)$ is generated by a set of polynomials all having degree six or less.
\end{itemize}
\end{theorem}

\begin{proof} Denote by    $\cZ(\I)$ the zero set of the ideal $\I$. First observe that $\L_\m \subseteq \cZ(\I)$ since
each vector $b\in \L_\m$ has squared length 32 and $S_{a,i}(b) = 0$ for all $a\in \L_\m$, $1\le i\le 23$, as explained in \sref{antip}.
Now suppose $p \in \cZ(\I)$.
Using \cref{slicedzonal}, we see that each $p$ has all real coordinates.
Since each sliced zonal polynomial $S_{a,i}$ vanishes at $p$, the inner product $p\cdot a$
is integral for every $a \in \L_{24}$ and, consquently, $p$ has integral inner product with
every Leech vector $b \in \Lambda_{24}$. It is well-known that $\Lambda_{24}$ is 
unimodular, so $p$ must belong to this lattice and, in view of its norm, belongs to 
$\L_{24}$, proving {\sl (i)}. 

For the proof of {\sl (ii)}, fix $a\in \L_\m$ and consider the Jacobian of the generating set $\G$ of $\I$, which has $1+23\cdot 98280$ rows and $24$ columns.
We must prove that this matrix has rank 24. To do so, we locate a $24\times 24$ invertible submatrix, which we will denote by $J(a)$.  To wit, consider
24 linearly independent elements $c^{(i)} \in X$, $1\le i\le \m$,   not including $\pm a$, and for each of the corresponding antipodal
pairs $\pm c^{(i)}$, abbreviate $c=c^{(i)}$ and select any sliced zonal polynomial
$S_{c,j}(Y)= (b^{(j)} \cdot Y)\prod_h (c\cdot Y - \omega_h)$ from $\G$ such that $b^{(j)} \cdot a \neq 0$. 
(This is possible since $b^{(j)} \cdot a  =0$ for all $1\le j\le 23$ implies that $a$ is a scalar multiple of 
$c$, a possibility we have explicitly excluded.) For the moment, denote these twenty-four polynomials by $F_1,\ldots, F_\m$.   We have 
$$ J(a) = \left[ \begin{array}{ccc}
\frac{ \partial F_1 }{ \partial Y_1} \big|_a  & \cdots  & \frac{ \partial F_1 }{ \partial Y_\m} \big|_a  \\
     \vdots   &   \ddots &   \vdots \\
\frac{ \partial F_\m }{ \partial Y_1} \big|_a & \cdots  & \frac{ \partial F_\m }{ \partial Y_\m}  \big|_a 
\end{array} \right] ~. $$
If $F_i(Y) = (b \cdot Y)  \prod_{h=1}^5 \left( c\cdot Y - \omega_h\right) $ for $c=c^{(i)}$ and some $b \perp c$,  $b \not\perp a$, then
$$ \frac{ \partial F_i }{ \partial Y_j}  = b_j  \prod_{h=1}^5 \left( c \cdot Y - \omega_h\right)  + c_j (b \cdot Y) \sum_{\ell = 1}^5  \prod_{h\neq \ell} \left( c\cdot Y - \omega_h\right)   $$
with 
$$ \frac{ \partial F_i }{ \partial Y_j} \bigg|_a = c_j (b \cdot a) \sum_{\ell = 1}^5  \prod_{h\neq \ell} \left( c\cdot a - \omega_h\right)  =
  c_j (b \cdot a)  \prod_{h\neq k} \left( c\cdot a - \omega_h\right)$$
  where $c\cdot a = \omega_k$. So the $i^{\rm th}$ row of $J(a)$  is simply
  $$ (b \cdot a)  \prod_{h\neq k} \left( c\cdot a - \omega_h\right) [c_1 \ c_2 \ \cdots \ c_\m]  ~.$$
Since the 24 scalars $(b \cdot a)  \prod_{h\neq k_i} \left( c^{(i)} \cdot a - \omega_h\right) $ are all non-zero, we have row equivalence
$$ J(a) \sim
\left[ \begin{array}{ccc}
\ \ \  &  c^{(1)} &  \ \ \ \\
   & \vdots  &    \\
\ \ \  &  c^{(24)} &  \ \ \ 
\end{array} \right] $$
which is, by design, an invertible $\m \times \m$ matrix.  It follows that the Zariski tangent space at $a$ (the common kernel of all differentials of all 
polynomials in the ideal) has dimension zero, so $a$ is indeed a simple zero of $\I$. 
Since all zeros are simple, it follows that $\I$ is a radical ideal. (See, e.g., Proposition 8 in Section 5.3 of \cite{iva}.)
This proves {\sl (ii)}.

We have exhibited non-trivial polynomials (the sliced zonals) of degree six and it is well-known that $\L_\m$ is a spherical 
11-design. So we have {\sl (iii)}. 
Finally, we have $\cI(\L_\m)  = \cI( \cZ(\I) ) = \Rad(\I) = \I $ by {\sl (i)}, Nullstellensatz, and {\sl (iii)}, respectively, giving {\sl (iv)}. \qed
\end{proof}

%%%%%%%%%%%%%%%%%%%%%%%%%%%%%%%%%%%%%%%%%%%%%%%%%%%%%%%%%
\section{Two new parameters}
\label{sec:twoparameters}

In this paper, we have considered several examples of point configurations on spheres and we have studied the ideal of 
polynomials which vanish on each of these finite algebraic sets. These examples were chosen because each corresponds to 
a cometric association scheme with extremal properties. It is no surprise then that, for these examples, the two parameters
we have proposed to study are equal. We now define these two parameters in general.

Let $\Omega_m$ denote the unit sphere in $\re^m$ and consider a finite non-empty set $X \subset \Omega_m$ (i.e., a
spherical code). For simplicity\footnote{In the case where $X$ is not full-dimensional, e.g., where $X$ is the 
set of columns of some primitive idempotent $E_1$ of an association scheme, our definition of ``trivial'' polynomial
must be modified to include combinations of linear functions vanishing on $X$.}, assume that $X$ spans $\re^m$.
The ideal $\cI(X)$ of polynomials in $m$ variables that vanish on $X$ contains the principal ideal $\langle \Nm \rangle$
of all multiples of the equation of the sphere. We define two parameters:
\begin{equation}
\label{eqn:gamma1}
\gamma_1 = \gamma_1(X) = \min \left\{ \deg F \bigm|  F \in \cI(X), \  \Nm \not| \  F \right\}
\end{equation}
and
\begin{equation}
\label{eqn:gamma2}
\gamma_2 = \gamma_2(X) = \min \left\{  \max_{F \in \G} \  \deg F \Bigm|  \langle \G \rangle = \cI(X)  \right\} ~ .
\end{equation}
Clearly $\gamma_1 \le \gamma_2$ and, in general, $\gamma_1 < \gamma_2$. For example, the points of $X$ may lie on two parallel
planes giving $\gamma_1=2$ yet $\gamma_2$ may still be made arbitrarily large. Yet, in all five spherical configurations considered above,
we have $\gamma_1=\gamma_2$. This also occurs for (natural Euclidean representations of) 
certain classical association schemes such as Hamming and Johnson schemes, where $\gamma_1=\gamma_2=2$ \cite{billsidealpaper}. 
We predict that the cometric association schemes 
for which $\gamma_1=\gamma_2$ are worthy of further study and perhaps can be classified.

\bigskip

\noindent {\bf Problem:} What can be said about $X$ when $\gamma_2(X)=2$? Can these sets be classified if we assume that $X$
comes from an association scheme?

\bigskip

For a more elementary class of examples, consider any set $X$ of  
$n$ points on the unit circle  $\Omega_2$. B\'{e}zout's Theorem implies that $\gamma_1 \ge \lceil n/2 \rceil$.
On the other hand, it is not hard to show $\gamma_2 =  \lceil n/2 \rceil$ as well.  Let us show this in the case 
where $n$ is even. Choose, in two different ways,  $n/2$ chords covering the $n$ points so that no chord from the second covering 
is parallel to any chord from the first. Let $F(Y)$ be a polynomial of degree $n/2$ which vanishes on the first set of chords and
let $G(Y)$ be a polynomial of degree $n/2$ which vanishes on the second set of chords.
 We obtain two polynomials each of which factors into $n/2$ linear factors. Clearly $\Nm$, $F$ and $G$ generate an ideal $\I$
with $\cZ(\I) \supseteq X$. But our choice of chords ensures that all zeroes of $\I$  are real and that each point of $X$ is a simple zero
of $\I$, ensuring $\langle \Nm, F, G \rangle = \cI(X)$. 

\bigskip

Both zonal polynomials and sliced zonal polynomials are non-trivial polynomials in our ideal. So we have the following

\begin{lemma}
\label{lem:gamma1}
For any spanning $X\subseteq \Omega_m$ with inner product set $\{ a\cdot b \mid a,b \in X\}$ of size $d+1$, $2 \le \gamma_1(X) \le d+1$. If
$X$ is antipodal, then $2\le \gamma_1(X) \le d$. \qed
\end{lemma}

\bigskip

As the polygons show, there can be no upper bound on $\gamma_2$ depending only on $m$. But in
the case where $X$ is the set of columns of the first idempotent $E_1$ of some cometric association 
scheme (in the space $\colsp E_1$ of dimension $m$), it follows from a
result of Martin and Williford \cite{mw} that there exists a finite upper bound on $\gamma_2(X)$, but no
explicit bound of this sort is known.

The vector space of polynomial functions on $\Omega_m$ of degree at most $k$ is studied in \cite{dgs}. 
The dimension of this space is denoted there by $R_k(1)$ and is equal to
$$ R_k(1) = \binom{ m + k - 1}{ m - 1 } +  \binom{ m + k - 2}{ m - 1 } ~ .$$
(see  \cite[Theorem~3.2]{dgs}). Elementary linear algebra then tells us that 
$$\gamma_1(X) \le \min \{ k \mid R_k(1) > |X| \} $$
and we have equality for $\L_\m$, $\E_8$ and $\E_6$.

\subsection{Cometric association schemes}
\label{sec:schemes}

In this final section, we outline some applications of our results and the tools introduced above to the theory of cometric 
($Q$-polynomial) association schemes. Basic background material on association schemes can be found in any of the following references:
\cite[p8]{del}, \cite[p52]{banito}, \cite[p43]{bcn}, \cite[p229]{godsil}.

Let $X$ be a finite set of size $v$ and let $\R = \{ R_0, \ldots, R_d\}$ be  a collection of $d+1$ binary relations on $X$.
The pair $(X,\R)$ is called a {\em symmetric $d$-class association scheme} provided the following four conditions hold:
\begin{itemize}
\item $R_0 = \{ (a,a) \mid a\in X\}$ is the identity relation on $X$;
\item the $R_i$ partition $X\times X$: $R_0 \cup \cdots \cup R_d = X\times X$ and $R_i \cap R_j = \emptyset$ whenever $i\neq j$;
\item there exist integers $p_{ij}^k$ ($0\le i,j,k \le d$) such that whenever $a,b\in X$ with $(a,b)\in R_k$,
$$ \left| \{ c \in X \ : \ (a,c)\in R_i , \ (c,b) \in R_j \} \right| = p_{ij}^k ~;$$
\item each $R_i$ is a symmetric relation: $R_i^\top = R_i$.
\end{itemize}
The elements of $X$ are called
{\em vertices}  and $R_i$ is called the {\em $i^{\rm th}$ adjacency} (or {\em basis}) {\em relation} of the scheme. In our discussion,
we have considered  pairs $(X,\R)$  arising from the following construction (and, in these five cases at least, 
we have obtained association schemes):

\bigskip

\noindent {\bf Construction 5.1:} Let $X$  be a finite subset of a sphere in $\re^m$ centered at the origin. Let
the inner product set of $X$ be written $\{ \omega_0,  \omega_1,\ldots, \omega_d \}$   where $a\cdot a = \omega_0$
for $a\in X$.  Partition  $X\times X$  into relations $R_i$ by inner product: $(a,b) \in R_h$ iff $a \cdot b = \omega_h$.

\bigskip

Each relation $R_h$ ($0\le h\le d$) of an association scheme gives us an undirected graph $(X,R_h)$ and we denote by 
$A_h$ the adjacency matrix of  this graph. It is well-known (see, e.g., \cite[Section~2.2]{bcn}) that the vector space 
$\BMA$ spanned by $\{A_0,\ldots, A_d\}$  is a commutative matrix algebra which is also closed under entrywise 
multiplication. This {\em Bose-Mesner algebra}  $\BMA$ admits a basis of 
mutually orthogonal idempotents $\{ E_0, E_1, \ldots, E_d\}$, these being the orthogonal projections onto the maximal common
eigenspaces of the matrices $A_h$. The Schur (or entrywise) product of any two of these idempotents belongs to $\BMA$ so there
exist scalars $q_{ij}^k$ ($0\le i,j,k\le d$), called {\em Krein parameters}, satisfying 
\begin{equation}
\label{eqn:qijk}
E_i \circ E_j = \frac{1}{|X|} \sum_{k=0}^d q_{ij}^k E_k ~. 
\end{equation}
An association scheme $(X,\R)$ is {\em cometric}  \cite[p58]{bcn} (or $Q$-{\em polynomial}) if there exists an ordering $E_0,\ldots, E_d$ of
its primitive idempotents such that the equations (\ref{eqn:qijk}) with respect to this ordering satisfy
\begin{itemize}
\item $q_{ij}^k = 0$ whenever $k > i+j$, and
\item $q_{ij}^k \neq 0$ whenever $k=i+j$.
\end{itemize}
In this case, the ordering $E_0,\ldots, E_d$ is called  a $Q$-{\em polynomial ordering}. (A given cometric association scheme 
may admit several such orderings.) The five examples considered in Sections \ref{sec:e8}, \ref{sec:leech} and  \ref{sec:icos}
all correspond to cometric association schemes.

\bigskip

Suppose $(X,\R)$ is any cometric (symmetric) association scheme with $Q$-polynomial ordering $E_0,E_1,\ldots,E_d$.
The column space of $E_j$ ($0\le j\le d$) is called 
the $j^{\rm th}$ {\em eigenspace} of the scheme and is denoted by $V_j$. We set $m_j = \rank E_j$.
One quickly sees that Equation  (\ref{eqn:qijk})  implies that  $E_k (E_i \circ E_j)=0$ whenever $q_{ij}^k=0$;
this identity extends nicely to entrywise products of arbitrary eigenvectors
due to the following theorem of 
Cameron, Goethals, and Seidel.  (See also Proposition II.8.3(i)
in~\cite{banito} and Theorem 2.3.2 in \cite{bcn}.)

\begin{theorem}[{\cite[Theorem~5.1]{cgs}}]
\label{thm:fund}
If $u \in V_i $ and $w \in V_j$ and $q_{ij}^{\ell} = 0$, then
$u \circ w$ is orthogonal to $V_\ell$  where $u \circ w$ denotes the entrywise product
of vectors $u$ and $w$.
\end{theorem}

Now with $m=\rank E_1$, choose any orthonormal basis for $V_1$,  arranging these vectors
in a matrix $U$ with $v$ rows and $m$ columns. Then we have $E_1 = U U^\top$ and we may consider the
spherical configuration $\bar{X}$  consisting of the rows of $U$. (It is well-known (e.g., \cite{mw}, but this was known
much earlier) that the association scheme $(X,\R)$ may be recovered from $\bar{X}$ via Construction 5.1.)

Every polynomial in $\cx[Y_1,\ldots,Y_m]$ which belongs to $\cI(\bar{X})$ gives us a polynomial
in $v$ variables --- $Z_1,\ldots,Z_v$, say --- which vanishes on each column of $E_1$; this is achieved by the simple linear 
change of variables $Y_h = \sum_a U_{ah} Z_a$. For example, if the cubic
$$ F(Y) = (Y \cdot \ell_1  - \alpha_1)  (Y \cdot \ell_2  - \alpha_2)  (Y \cdot \ell_3  - \alpha_3) $$
vanishes on each row of $U$, then we have 
$$ \left( U \ell_1 - \alpha_1\one \right) \circ  \left( U \ell_2 - \alpha_2\one \right)  \circ \left( U \ell_3 - \alpha_3\one \right) = 0$$
in $\cx^v$ where $\one$ denotes the vector of all ones. But $\rank(U)=m$, so each vector $\ell_i$ of length $m$ may be expressed $\ell_i = U^\top w_i$ for some
vector $w_i$ of length $v$. Thus we have
$$ \left( E_1 w_1 - \alpha_1\one \right) \circ  \left( E_1 w_2 - \alpha_2\one \right)  \circ \left(  E_1 w_3 - \alpha_3\one \right) = 0$$
and this gives us a cubic polynomial which vanishes on each row of $E_1$. If we write $u_i = E_1 w_i -\alpha_i \one$,
then we have three vectors $u_1, u_2, u_3$ in $V_0+ V_1$ whose entrywise product is zero. 

For example, up to scalar, we may take our icosahedron to consist of the twelve points
$$ (\pm1, \pm \varphi, 0), \quad (0, \pm1, \pm \varphi), \quad (\pm \varphi, 0, \pm1) $$
where $\varphi = \frac{1}{2}(1+\sqrt{5})$ and the sliced zonal polynomial 
$$ (Y_1 + \varphi Y_2 - \varphi) (Y_1 + \varphi Y_2  + \varphi) Y_3 $$
lifts (under any ordering where $(1,\varphi,0),(0,1,\varphi),(\varphi,0,1)$ are the first, second and third vertices, respectively) to
$$ \frac{1}{2}[ Z_1 - \varphi][Z_1+\varphi]\left[ (1-\varphi)Z_1 + Z_2 + (2-\varphi)Z_3 \right] $$
which vanishes on each row of $\sigma E_1$ for $\sigma=10+2\sqrt{5}$. In particular, if we let 
$u \in \re^{12}$ with $u_b=(E_1)_{b,1}-\varphi$ and
$w \in \re^{12}$ with $w_b=(E_1)_{b,1}+\varphi$, then the entrywise product $u \circ w$ has only two nonzero entries and
belongs to $V_0+V_1+V_2$ by \tref{fund}.
\bigskip

In \cite{billsidealpaper}, we investigate this sort of structure further and discuss applications. For example, we exhibit connections
between the ideal $\cI(X)$ and completely regular codes in cometric distance-regular graphs, we consider the Hilbert
series of this ideal and its relationship to a conjecture of Bannai and Ito, and we explore the role of this ideal in duality
of association schemes. 

Rather than giving full details here, we are content to 
demonstrate these ideas by  considering one more, quite simple, family of examples in detail: the complete bipartite graphs $K_{n,n}$.
These are obviously strongly regular graphs and so give us 2-class cometric association schemes.

%\begin{example} 
Let $X$ be a spherical code of $2n$ points in $\re^{2n-2}$ with angles 
$0$, $\pi/2$ and $\arccos(\frac{-1}{n-1})$ consisting  of  $n$-simplices arranged in a pair of orthogonal 
$(n-1)$-dimesional subspaces. Then the three angles give us the distance relations  $\{R_0,R_1,R_2\}$ of 
the association scheme of the complete bipartite graph. Our eigenspace dimensions are
$m_0=1$, $m=m_1=2n-2$ and $m_2=1$. Scaled by $\sqrt{m/n}$ for convenience, our spherical drawing may be written
$$ a^{(i)} = \left( \underbrace{0,\ldots,0}_{i-2}, \frac{ 1-i }{ \sqrt{ \binom{i}{2} } } ,  \frac{ 1}{ \sqrt{ \binom{i+1}{2} } } ,  \frac{ 1}{ \sqrt{ \binom{i+2}{2} } } , 
\ldots,  \frac{ 1}{ \sqrt{ \binom{n}{2} } } , \underbrace{0,\ldots, 0}_{n} \right) $$
$$ b^{(i)} = \left(  \underbrace{0,\ldots, 0}_{n} , \underbrace{0,\ldots,0}_{i-2}, \frac{ 1-i }{ \sqrt{ \binom{i}{2} } } ,  \frac{ 1}{ \sqrt{ \binom{i+1}{2} } } ,  \frac{ 1}{ \sqrt{ \binom{i+2}{2} } } , 
\ldots,  \frac{ 1}{ \sqrt{ \binom{n}{2} } } \right) $$
($1\le i \le n$, and ignoring some terms for $i=1,2$) with inner products
$$ \omega_0 = a^{(i)} \cdot a^{(i)} = 2-\frac{2}{n}, \quad \omega_1 =  a^{(i)} \cdot b^{(j)} = 0 \ \ (\text{any} \ i,j), \qquad  a^{(i)} \cdot a^{(j)} = -\frac{2}{n} \ \ (i\neq j) $$
(and the same with $a$ and $b$ symbols swapped).

Every edge of $K_{n,n}$ is a completely regular code \cite[p346]{bcn}, giving us an equitable partition with quotient matrix 
$\left[ \begin{array}{cc} 1 & n-1 \\ 1 & n-1 \end{array} \right]$ having eigenvalues $n$ and $0$.  Each such edge $\{a^{(i)},b^{(j)} \}$
provides us with a quadratic polynomial 
$$ F_{i,j}(Y) = \left( a^{(i)} \cdot Y + b^{(j)} \cdot Y - \left(2-\frac{2}{n} \right) \right)\left( a^{(i)} \cdot Y + b^{(j)} \cdot Y + \frac{2}{n} \right)$$
in our ideal and an easy induction proof shows that these $n^2$ quadratic polynomials $\{ F_{i,j} \mid 1 \le i,j \le n\}$ generate our ideal
$\cI(X)$.

In this representation of $X$, we may easily describe the ideal of leading terms for  $K_{n,n}$, i.e., the polynomial ideal generated by
all leading terms of polynomials in $\cI(X)$ with respect to some total ordering. If we choose {\it grevlex}  as our monomial 
ordering (see, e.g., \cite[p58]{iva},\cite[p331]{df}), 
 then every degree two monomial {\em except} $Y_{m}^2$
belongs to the ideal of leading terms {\sc lt}$( \cI(X) )$ as does $Y_m^3$. So the only monomials \underline{not} belonging to the ideal 
{\sc lt}$( \cI(X) )$ are 
$$1, \ \ \ Y_1, \ Y_2, \ \ldots, \ Y_m, \ \ \ Y_m^2 $$
and the number of such monomials of degree $j$ is $\rank E_j$, as expected; our Hilbert series (with respect to any total degree ordering)
is then $1 + (2n-2)t + t^2$.

\bigskip

Finally, we can also find generators for our ideal using the dual association scheme. For this, we employ a different, but equivalent drawing.

It is well-known that $K_{n,n}$ is a translation association scheme for the group $\ints_{2n}$: vertices may be labeled by 
$0,1,\ldots, 2n-1$ with two vertices adjacent if and only if their difference modulo $2n$ is odd.
So let $\zeta$ be a complex primitive $2n^{\rm th}$ root of unity and let $\hat{U}$ be the matrix with rows and columns indexed 
by $G= \ints_{2n}$ with $(a,b)$-entry $\hat{U}_{a,b} = \zeta^{ab}$ where exponents are computed modulo $2n$. Each column 
of $\hat{U}$ is a character of $G$ and these are well-known to provide a basis of eigenvectors for the Bose-Mesner algebra
$\BMA = \langle A \rangle$  of the translation scheme $K_{n,n}$ (here, $A$ is the adjacency matrix of this graph). 

Let $U$ be the submatrix obtained by deleting the columns indexed by $0$ and $n$; mapping 
the vertex set $G$ to $\cx^{2n-2}$ by sending vertex $a$ to row $a$ of this matrix $U$ gives us a spherical configuration,
which we will also call $X$, of
$2n$ points on a sphere of radius $\sqrt{2n-2}$ with Hermitian inner products $2n-2$ ($a=b$), $0$ ($a-b$ odd) and $-2$ ($a-b$ even).
For consistency, we label the columns of $U$ by $C=\{1,\ldots, n-1, n+1,\ldots, 2n-1\}$ and work with indeterminates $Y_i$ ($1\le i < 2n$, $i\neq n$).
 
Let $\chi_i$ denote the $i^{\rm th}$ column of $\hat{U}$; then $A \chi_i = \pm n \chi_i$ for $i=0,n$ and $A\chi_i = 0 \chi_i$ for $i\in C$. So the
eigenspaces of the association scheme are simply  $V_0 = \spn \{ \chi_0 \}$, $V_2 = \spn \{ \chi_n \}$, and 
$$ V_1 = \spn \{ \chi_i \mid i \neq 0,n \} .$$

The dual association scheme  \cite[Section~2.6]{del} 
has the group $\ints_{2n}^*$ of characters as vertices and relations
$$ R_j^* = \left\{  (\phi,\psi) \mid \psi = \phi \circ \chi \ \text{for some character} \ \chi \in V_j \right\}. $$
So, in our case with $d=2$, $m_1=2n-2$ and $m_2=1$, the dual association scheme is the one arising from a complete
multipartite graph with $n$ parts of size two, $K_{2,\ldots, 2}$.  We now find a  generating set for the ideal $\cI(X)$ 
from the  adjacency relation of this graph:
$R_1^* = \left\{ (\phi, \phi\circ \chi_j) \mid \phi \in \ints_{2n}^*, \ j\neq 0,n \right\}$.
First note that each $\chi_j$ ($1\le j < n$) gives us a quadratic polynomial in our ideal since the group equation $\chi_i \circ \chi_j = \chi_{i+j}$
gives us   $\chi_j \circ \chi_{2n-j} =\chi_0= \one$; that is,  $Y_j Y_{2n-j} - 1 \in \cI(X)$. We next use these quadratics to reduce polynomials
coming from the closed walks in the graph   $K_{2,\ldots, 2}$.

Each cycle $\psi_1,\psi_2,\ldots, \psi_k, \psi_1$ of length $k$ in the graph  $K_{2,\ldots, 2}$ gives rise to a polynomial of degree $k$ 
in the ideal $\cI(X)$. For example, if $\psi_1,\psi_2,\psi_3,\psi_1$ is any cycle of length three in $K_{2,\ldots, 2}$, 
then there must exist characters $\chi_{i_1},\chi_{i_2},\chi_{i_3} \in V_1$ with 
$$  \psi_2 = \psi_1 \circ \chi_{i_1}, \quad   \psi_3 = \psi_2 \circ \chi_{i_2}, \quad   \psi_1 = \psi_3 \circ \chi_{i_3}, $$
so that 
$$  \psi_1 = \psi_1 \circ \chi_{i_1} \circ \chi_{i_2} \circ \chi_{i_3} $$
and $ \chi_{i_1} \circ \chi_{i_2} \circ \chi_{i_3}$ is equal to the all-ones vector since $\psi_1$ has no zero entries. 
So the entrywise product of three columns of $U$ is equal to the vector of all ones. This gives us the polynomial
$$ Y_{i_1} Y_{i_2} Y_{i_3} - 1 \in \cI(X) .$$
Since $Y_{i_1} Y_{2n-i_1} \equiv 1$ modulo $\cI(X)$, we also infer that the quadratic $Y_{i_2} Y_{i_3} - Y_{2n-i_1}$ belongs to $\cI(X)$ 
as well. It is not hard to see that every directed cycle in $K_{2,\ldots,2}$ is an integer sum of 3-cycles. So the ideal generated by these 
cubics (or quadratics) corresponding to the 3-cycles
contains the polynomials determined in this way by all cycles in $K_{2,\ldots,2}$. For example, provided $j-i \neq 0,n$, the quadrilateral
$\chi_0, \chi_i, \chi_n, \chi_j, \chi_0$ in  $K_{2,\ldots, 2}$ gives us the degree four polynomial
$$ G(Y) = Y_{2n-j} Y_{j-n} Y_{n-i} Y_i - 1 \in \cI(X).$$
But this decomposes into two triangles with corresponding cubics $F_1(Y) = Y_{2n-j}  Y_{j-i} Y_i - 1$ and 
$F_2(Y) = Y_{i-j} Y_{j-n} Y_{n-i} - 1$ in the ideal, where (modulo $Y_{i-j}Y_{j-i} - 1$) we have
$$ G(Y) = F_1(Y) F_2(Y)  +  F_1(Y) +  F_2(Y) ~. $$
It is then not hard to show that the quadratics $Y_i Y_j - Y_{i+j}$ also form a generating set for our ideal (any vector satisfying 
these relations must arise from a character), when $X$ is 
represented in these coordinates.
%\end{example}

%%%%%%%%%%%%%%%%%%%%%%%%%%%%%%%%%%%%%%%%%%%%%%%%%%%%%%%%%
\section*{Acknowledgments}

This work was supported by the National Security Agency. Summer support for both authors is
gratefully acknowledged. We also thank Jonathan Godbout for useful suggestions. 
The first author benefited from discussions with Eiichi Bannai, Shuhong Gao, Chris Godsil, Felix Lazebnik, 
John Little, Akihiro Munemasa, Frank Sottile, Hajime Tanaka, Paul Terwilliger and Jason Williford over 
several years. The final preparation of this paper occurred while the first author was visiting the University 
of Delaware and he thanks the Department of Mathematical Sciences for their hospitality.
Finally, the authors are grateful to the referee for a careful reading of the manuscript and for suggesting several improvements.

\end{document}